\documentclass[11pt]{amsart}

\usepackage{xy}
\usepackage{array}
\usepackage{amssymb}
\usepackage{color}
\xyoption{all}

\newtheorem{lemma}{Lemma}[section]
\newtheorem{theorem}[lemma]{Theorem}
\newtheorem{corollary}[lemma]{Corollary}

\newtheorem{proposition}[lemma]{Proposition}

\newtheorem{remark}[lemma]{Remark}

\newcommand{\Hom}{\operatorname{Hom}}

\newcommand{\Ext}{\operatorname{Ext}}

\newcommand{\End}{\operatorname{End}}
\newcommand{\op}{{\textit{op}}}
\renewcommand{\ker}{\operatorname{ker}}
\newcommand{\coker}{\operatorname{coker}}
\newcommand{\coim}{\operatorname{coim}}
\newcommand{\im}{\operatorname{im}}
\newcommand{\add}{\operatorname{add}}
\newcommand{\module}{\operatorname{mod}}

\newcommand{\U}{\mathcal U}
\newcommand{\V}{\mathcal V}
\newcommand{\C}{\mathcal C}
\newcommand{\D}{\mathcal D}
\newcommand{\X}{\mathcal X}
\renewcommand{\S}{\mathcal S}

\newcommand{\CX}{\C/{\X_T}}
\newcommand{\CXR}{(\C/{\X_T})_{\R}}

\newcommand{\A}{\mathcal A}
\newcommand{\R}{\mathcal R}

\begin{document}

\title[From triangulated categories to module categories II]{From
  triangulated categories to module categories via localisation II:
  calculus of fractions}

\author[Buan]{Aslak Bakke Buan}
\address{
Department of Mathematical Sciences,
Norwegian University of Science and Technology,
7491 Trondheim,
NORWAY 
}
\email{aslakb@math.ntnu.no}

\author[Marsh]{Robert J. Marsh}
\address{School of Mathematics \\
University of Leeds \\
Leeds LS2 9JT \\
England
}
\email{marsh@maths.leeds.ac.uk}


\begin{abstract}
We show that the quotient of a Hom-finite triangulated category $\C$ by the kernel of the
functor $\Hom_{\C}(T,\,-)$, where $T$ is a rigid object, is preabelian. We further show that the
class of regular morphisms in the quotient admit a calculus of left and right fractions. It
follows that the Gabriel-Zisman localisation of the quotient at the class of regular morphisms
is abelian. We show that it is equivalent to the category of finite dimensional modules
over the opposite of the endomorphism algebra of $T$ in $\C$.
\end{abstract}

\date{20 September 2011}


\maketitle

\section*{Introduction}

Let $k$ be a field and $\C$ a skeletally small, triangulated $\Hom$-finite $k$-category which
is Krull-Schmidt and has Serre duality.
A standard example of such a category is the bounded derived category
of finite dimensional modules over a finite dimensional algebra
of finite global dimension (see~\cite{happel}).
In this case, the triangulated category is obtained from the abelian category
of modules by Gabriel-Zisman (or Verdier) localisation of the quasi-isomorphisms
in the bounded homotopy category of complexes of modules.

Here, our approach is the other way around.
Given a triangulated category $\C$
as above, we are interested in gaining information about related abelian categories.
We are particularly interested in the module categories over (the opposites of)
endomorphism algebras of objects in $\C$.
An object $T$ in $\C$ satisfying $\Ext^1(T,T)=0$ is known as a
\emph{rigid} object. In this case it is known~\cite{bm10} that the category of finite
dimensional modules over $\End(T)^{\op}$ can be obtained as a Gabriel-Zisman localisation
of $\C$, formally inverting the class $\S$ of maps which are inverted by the functor
$\Hom_{\C}(T,\,-)$. However, the class $\S$ does not admit a calculus of left or right
fractions in the sense of~\cite[Sect.\ I.2]{gz} (see also~\cite[Sect.\ 3]{krause}).

If $T$ is a cluster-tilting object then, by a result of Koenig-Zhu~\cite[Cor.\ 4.4]{kz},
the additive quotient $\C/\Sigma T$, where $\Sigma$ denotes the suspension functor of $\C$,
is equivalent to $\module \End_{\C}(T)^{\op}$
(see also~\cite[Prop.\ 6.2]{iy} and~\cite[Sect.\ 5.1]{kr}; the case where $\C$
is $2$-Calabi-Yau was proved in~\cite[Prop. 2.1]{kr}, generalising~\cite[Thm. 2.2]{bmr}).
However, when $T$ is rigid, this is no longer the case in general. It is natural to
consider instead the quotient $\CX$ where $\X_T$ is the class of objects in $\C$
sent to zero by the functor $\Hom_{\C}(T,\,-)$, since, in the cluster-tilting case,
$\X_T=\add \Sigma T$. However, one does not obtain the module category this
way, since in general $\CX$ is not abelian.

Our approach here is to show first that $\CX$ is preabelian, using some arguments
generalising those of Koenig-Zhu~\cite{kz}. This means that, in addition to
$\CX$ being an additive category, every morphism in $\CX$ has a kernel and a
cokernel.
This category in general possesses regular morphisms which are not isomorphisms
(i.e.\ morphisms which are both monomorphisms and epimorphisms but which do not
have inverses), so it cannot be abelian in general.

However, we show that $\CX$ does have a nice property.
It is \emph{integral}, i.e.\ the pullback of any epimorphism (respectively,
monomorphism), is again an epimorphism (respectively, monomorphism).
This allows us to apply a result of Rump~\cite[p173]{r01} which implies that
$\CXR$, the localisation of the category $\CX$ at the class $\R$ of regular
morphisms, is abelian. We assume that $\C$ is skeletally small to ensure
that the localisation exists. Furthermore, by the same reference,
the class $\R$ admits a calculus of left and right fractions.

We go on to show that the projective objects in $\CXR$ are, up to isomorphism,
exactly the objects induced by the objects in the additive subcategory of $\C$ generated by $T$. This implies our main result:
\vskip 0.3cm
\noindent \textbf{Theorem.}
\emph{
Let $\C$ be a skeletally small, Hom-finite, Krull-Schmidt triangulated category with Serre duality, containing a rigid object $T$. Let $\X_T$ denote the class of objects $X$ in $\C$ such that $\Hom_{\C}(T,X)=0$.
Let $\R$ denote the class of regular morphisms in $\CX$.
Then $\R$ admits a calculus of left and right fractions.
Let $\CXR$ denote the localisation of $\CX$ at $\R$.
Then $$\CXR \simeq \module \End_{\C}(T)^{\op}.$$
}

Let $\S$ denote the class of maps in $\C$ which are inverted by
$\Hom_{\C}(T, -)$, and let $\underline{\S}$ denote the image of this
class in $\CX$. Then $\R = \underline{\S}$, and the localisation functor
$L_{\S} \colon \C \to \C_{\S}$ factors through $\CX$.
The main result of~\cite{bm10} was the construction of
an equivalence $G$ from $\C_{\S}$  to
$\module \End_{\C}(T)^{\op}$,
such that $\Hom_{\C}(T, - )=G L_{\S}$.
Our theorem above can be seen a refinement of this.
It was noted in~\cite{bm10} that $\S$ does not admit a calculus of
left or right fractions, thus we can observe that the advantage of passing
first to the quotient $\CX$ is that the subsequent localisation does then
admit such a calculus.

We note that~\cite{nakaoka} contains results obtaining abelian
categories as subquotients of triangulated categories; we give an
explanation of the relationship between the results obtained here
and those in~\cite{nakaoka} in Section~\ref{s:cotorsionpairs}.
We also remark that A. Beligiannis has recently informed us that,
in subsequent work using a different approach,
he has been able to generalise our main result
to the case of a functorially finite rigid subcategory.

There are interesting parallels between our approach here and the
construction of the derived category of an abelian category $\A$.
We follow~\cite{keller},~\cite[III.2,\, III.4]{gm}.
The derived category of $\A$ can be defined (following Grothendieck) as the
Gabriel-Zisman localisation of the category $C(\A)$ of complexes over $\A$
at the class of quasi-isomorphisms. This class does not in general admit a
calculus of left and right fractions.
However, the more commonly used construction (due to Verdier) of the derived
category involves passing first to the homotopy category $K(\A)$.
Then $C(\A)$ is a Frobenius category and
$K(\A)$ is the corresponding stable category, hence a quotient of $C(\A)$.
Then the class of quasi-isomorphisms in $K(\A)$ admits a calculus
of left and right fractions. Localising at this class gives rise to the
derived category of $\A$.

In Section~\ref{s:notation} we set-up the context in which we work.
In Section~\ref{s:preabelian}, we recall the definitions of semi-abelian and
integral categories and some results of Rump~\cite{r04,r01} which will be
useful.
In Section~\ref{s:cxproperties}, we prove that $\CX$ is
integral.
In Section~\ref{s:localisation}, we recall the Gabriel-Zisman theory of
localisation and calculi of fractions and also how it can be applied (following
Rump~\cite[Sect.\ 1]{r01}) to the case of the regular morphisms in an integral
category.
In Section~\ref{s:mainresult}, we apply this to $\CXR$ to show that it is
abelian.
By classifying the projective objects in $\CXR$, we deduce the main result.
In Section~\ref{s:cotorsionpairs}, we explain the relationship of the results here
to work of Nakaoka~\cite{nakaoka}.
In Section~\ref{s:diagrams}, we explain the relationship between our main result and the
results in~\cite{bm10}.

\section{Notation}
\label{s:notation}

We first set up the context in which we work and define some notation.
Let $k$ be a field and $\C$ be a skeletally small, triangulated, $\Hom$-finite, Krull-Schmidt $k$-category with suspension functor $\Sigma$. We need the skeletally small assumption
to ensure that the localisations we need exist.
We assume that $\C$ has a Serre duality, i.e.\ an autoequivalence
$\nu \colon \C\rightarrow \C$
such that $\Hom_{\C}(X,Y)\simeq D\Hom_{\C}(Y,\nu X)$ (natural in $X$ and $Y$)
for all objects $X$ and $Y$ in $\C$, where
$D$ denotes the duality $\Hom_k(-,k)$.
Let $T$ be a rigid object in $\C$ and set $\Gamma=\End_{\C}(T)^{\op}$.

For a full subcategory $\X$ of $\C$, let
$$\X^{\perp} = \{ C \in \C \mid \Ext^1(X,C) = 0 \text{ for each } X \in \X \},$$
and define $^{\perp}\X$ dually. 
For an object $X$ in $\C$, let $\add X$ denote its additive closure, and
let $X^{\perp} = (\add X)^{\perp}$.
A rigid object $T$ is called \emph{cluster-tilting} if $\add T = T^{\perp}$.
Let $\X_T = (\Sigma T)^{\perp}$. 

We also recall the triangulated version of Wakamatsu's Lemma;
see e.g.~\cite[Section 2]{iy}.

\begin{lemma}\label{wakamatsu}
Let $\X$ be an extension-closed subcategory of a triangulated category
$\C$.
\begin{itemize}
\item[(a)] Suppose that $X \to C$ is a minimal right $\X$-approximation
of $C$ and $\Sigma^{-1} C \to Y \to X \to C $ a completion to a triangle.
Then $Y$ is in $\X^{\perp}$, and the map
 $\Sigma^{-1} C \to Y$ is a left $\X^{\perp}$-approximation of $\Sigma^{-1} C$.

\item[(b)] Suppose that $C \to X$ is a minimal left $\X$-approximation
of $C$ and  $\Sigma^{-1}Z \to C \to X \to Z \to \Sigma C $ a completion
to a triangle.
Then $Z$ is in $^{\perp}\X$, and the map
 $Z \to \Sigma C$ is a right $^{\perp}\X$-approximation of $\Sigma C$.
\end{itemize}
\end{lemma}

Using this, we obtain:

\begin{lemma}\label{cvf}
Let $T$ be a rigid object in $\C$.
Then the subcategory $\X_T$ of $\C$ is functorially finite.
\end{lemma}

\begin{proof}
This follows from combining Wakamatsu's Lemma (Lemma~\ref{wakamatsu})
with the existence of Serre duality.
\end{proof}

\section{Preabelian categories} \label{s:preabelian}

Recall that an additive category $\A$ is said to be \emph{preabelian} if every
morphism has a kernel and a cokernel. In this section we shall recall some of the
theory of preabelian categories that we need in order to study $\C/\X_T$.
A morphism is said to be \emph{regular} (or a bimorphism) if it is both an epimorphism and a
monomorphism. 

According to~\cite[Sect. 1]{r01}
a preabelian category is called \emph{left semi-abelian} (respectively,
\emph{right semi-abelian}) if every morphism $f$
has a factorisation of the form $ip$ where $p$ is a cokernel and $i$ is
a monomorphism (respectively, where $p$ is an epimorphism
and $i$ is a kernel); see~\cite[Sect. 1]{r01}, where it is pointed out that
in the left semi-abelian case $p$ is necessarily $\coim(f)=\coker(\ker(f))$ and in the right
semi-abelian case $i$ is necessarily $\im(f)=\ker(\coker(f))$.
A preabelian category is said to be \emph{semi-abelian} if it is both left and
right semi-abelian.

We remark that pullbacks and pushouts always exist in a preabelian category.
For the pullback of maps $c \colon B\to D$ and $d \colon C\to D$, we can take the kernel of
the map $B\coprod C\to D$ whose components are $c$ and $-d$, obtaining a
pullback diagram:

\begin{equation} \label{e:pullback}
\xymatrix{
A \ar[r]^a \ar[d]_b & B \ar[d]^c \\
C \ar[r]_d & D
}
\end{equation}

There is a dual construction for the pushout. 

We recall the following characterisation of semi-abelian categories:

\begin{proposition} \cite[Prop.\ 1]{r01} \label{p:semi-abelianlift}
Let $\A$ be a preabelian category. Then $\A$ is left semi-abelian if and only if,
in any pullback diagram as above, $a$ is an epimorphism whenever $d$ is
a cokernel.
\end{proposition}

A dual characterisation in terms of pushout diagrams holds for right
semi-abelian categories.

A preabelian category is said to be \emph{left integral} provided that,
in any pullback diagram as above,
$a$ is an epimorphism whenever $d$ is an epimorphism.
A dual definition involving pushouts is used to define right integral
categories. A preabelian category which is both left and right integral is said to be \emph{integral}.

The following then follows from Proposition~\ref{p:semi-abelianlift}.

\begin{proposition} \cite[Cor.\ 1]{r01} \label{p:integralimpliessemi-abelian}
Any left integral (respectively, integral) category is left semi-abelian
(respectively, semi-abelian).
\end{proposition}

We recall the following two results from~\cite{r01}.

\begin{lemma} \cite[Lemma 1]{r01} \label{l:pullbackofmono}
Let $\A$ be a preabelian category.
In a pullback diagram~\eqref{e:pullback}, whenever $d$ is a monomorphism,
$a$ is a monomorphism.
\end{lemma}

We also recall the following (which also includes a dual statement involving
pushouts which we will not need here).

\begin{proposition} \cite[Prop.\ 6]{r01} \label{p:integralcharacterisation}
Let $\A$ be a semi-abelian category. Then the following are equivalent.
\begin{enumerate}
\item[(a)] The category $\A$ is integral.
\item[(b)] For any pullback diagram~\eqref{e:pullback}, $a$ is regular
whenever $d$ is regular.
\end{enumerate}
\end{proposition}

Finally, we note the following, which is easy to show using the definitions.

\begin{lemma}\label{weakandepi}
Let $h$ be a map in an additive category which is a weak cokernel of a map
$g$ and an epimorphism. Then $h$ is a cokernel of $g$.
\end{lemma}

\section{Properties of $\CX$}
\label{s:cxproperties}

In this section, we consider the factor category $\CX$. The objects in $\CX$ are
the same as those in $\C$. For objects $X,Y$ in $\C$,
$\Hom_{\CX}(X,Y)$ is given by $\Hom_{\C}(X,Y)$ modulo morphisms
factoring through $\X_T$. We denote the image of a morphism $f$ in $\C$ by $\underline{f}$.
Note that, since $\C$ is $k$-additive, so is $\CX$.

We will show the following result, which can be regarded as a
generalisation of \cite[Theorem 3.3]{kz}. Our proof is inspired by the
proof in \cite{kz}.

\begin{theorem}\label{preab}
The factor category $\C/ \X_T$ is preabelian.
\end{theorem}

In order to prove Theorem~\ref{preab}, we will need the following lemmas.

\begin{lemma}\label{chasing}
Consider a commutative diagram
$$
\xymatrix{
A \ar[r]^{\alpha} \ar[d]^{\delta_1} & B \ar[r]^{\beta} \ar[d]^{\delta_2}  &  C
\ar[r]^{\gamma} \ar[d]^{\delta_3}  & \Sigma A \ar[d]^{\Sigma \delta_1} \\
A' \ar[r]^{\alpha'} & B' \ar[r]^{\beta'} & C' \ar[r]^{\gamma'}  & \Sigma A' 
}
$$
in a triangulated category, where the rows are triangles. 
\begin{itemize}
\item[(a)] If the composition $\delta_2 \alpha$ vanishes, then
there are maps $\epsilon_1 \colon C \to B'$ and $\epsilon_2 \colon
\Sigma A \to C'$, such that $\delta_3 = \beta' \epsilon_1 + \epsilon_2
\gamma$.
\item[(b)] If the composition $\gamma' \delta_3 $ vanishes, then there
are maps $\phi_1 \colon B \to A'$ and $\phi_2 \colon C \to B'$ such
that $\delta_2 = \phi_2 \beta + \alpha' \phi_i$.    
\end{itemize}
\end{lemma}

For a map $f \colon X \to Y$ in $\C$, consider the
triangle $\Sigma^{-1}  Z\overset{h}{\to} X \overset{f}{\to} Y \overset{g}{\to} Z$.

\begin{lemma}\label{diagonals}
\begin{itemize}
\item[(a)] The map $\underline{f} \colon X \to Y$ is a monomorphism  if and only if $\underline{h} =0$.
\item[(b)] The map $\underline{f} \colon X \to Y$ is an epimorphism  if and only if $\underline{g} =0$.
\item[(c)] The map $\underline{f} \colon X \to Y$ is regular if and only if
$\underline{g} =0 = \underline{h}$.
\end{itemize}
\end{lemma}

\begin{proof}
(c) follows by definition from (a) and (b). We prove only (b), the
proof of (a) being dual.

We consider first the case when $Z$ is in $\X_T$.
We then need to show that $\underline{f}$ is an epimorphism in $\C/ \X_T$.

Let $p \colon Y \to M$ be a map such that $\underline{p} \underline{f}
= 0$. Then there is an object $U'$ in $\X_T$, and a commuting square:
$$
\xymatrix{
X \ar[r]^{f} \ar[d] & Y \ar[d]^{p} \\
U' \ar[r] & M
}
$$

By Lemma~\ref{cvf}, a minimal right $\X_T$-approximation
$U \to M$ exists. The
map $U' \to M$ factors through 
$U \to M$, hence there is also a commutative square

$$
\xymatrix{
X \ar[r]^{f} \ar[d]_{q} & Y \ar[d]^{p} \\
U \ar[r]^{f'} & M
}
$$
which we extend to a commutative diagram

$$
\xymatrix{
\Sigma^{-1} Z \ar[r] \ar[d] &  X \ar[r]^{f} \ar[d]^{q} & Y \ar[d]^{p} \ar[r]
& Z \ar[d] \\
\Sigma^{-1} N \ar[r] & U \ar[r]^{f'} & M \ar[r] & N  
}
$$
where the rows are triangles. 
By Wakamatsu's Lemma (Lemma~\ref{wakamatsu}), we have that $\Sigma^{-1}N$ is in 
${\X_T}^{\perp}$. Therefore the map $Z \to N$ is zero, using that $Z$ is
by assumption in $\X_T$. By commutativity, 
the composition $\Sigma^{-1} Z \to X \to U$ vanishes.
Hence, we have by Lemma \ref{chasing},
that there are maps $v_1 \colon Y \to U$, and $v_2 \colon Z \to M$,
such that $p =   f' v_1 + v_2 g$. Hence $p$ factors through $U \amalg
Z$, which is in $\X_T$, so we have $\underline{p}= 0$.

Now consider the general case, so assume $g$ factors through an object
$V$ in $\X_T$ and consider the induced commutative diagram
$$
\xymatrix{
\Sigma^{-1} V \ar[r] \ar[d] & N \ar[r]^{f'} \ar[d]^{r}  & Y \ar@{=}[d]  \ar[r] &
V \ar[d] \\
\Sigma^{-1} Z \ar[r]^{h} & X \ar[r]^{f} & Y \ar[r]^{g} & Z  \\
}
$$
where the rows are triangles.
Now $f' = f r$ and hence $\underline{f'} =  \underline{f}
\underline{r}$. Note that since $V$ is in $\X_T$, we have that 
$\underline{f'}$ is an epimorphism. It follows that
$\underline{f} \underline{r}$, and hence $\underline{f}$, is an
epimorphism.

Conversely, if $\underline{f}$ is an epimorphism then, since $gf=0$ we have
$\underline{g}\underline{f}=0$, so $\underline{g}=0$.
\end{proof}

\begin{lemma} \label{kernelcokernel}
For any map $f \colon X \to Y$ in $\C$, the map 
$\underline{f} \colon X \to Y$ has a kernel and a cokernel.
\end{lemma}

\begin{proof}
We construct a cokernel of $\underline{f}$. The construction of a
kernel is dual. 

Consider a minimal right $\add T$-approximation $a \colon T_0 \to
X$. Compose this with $f$,  and complete the composition $fa$ to a triangle 
\begin{equation}\label{c-tria}
T_0 \to Y \overset{c}{\to} M \to \Sigma T_0
\end{equation}

By the octahedral axiom, there is a commutative diagram

$$
\xymatrix{
T_0 \ar[r] \ar[d]_{a} & Y \ar[r]^{c} \ar@{=}[d] & M \ar[r] \ar[d]^{b} & \Sigma T_0
\ar[d]^{\Sigma a} \\
X \ar[r]_{f} & Y \ar[r]  & Z \ar[r] & \Sigma X \\
}
$$
We claim that $\underline{c}$ is a cokernel for $\underline{f}$.

Consider a map $p \colon Y \to N$, such that $\underline{p}
\underline{f} = 0$. Assume $pf$ factors through an object $U$ in
$\X_T$, so there is a commuting square

$$
\xymatrix{
X \ar[r]^{f} \ar[d]& Y \ar[d]^{p} \\
U \ar[r]^{r} & N  
}
$$

Extend this to a commuting diagram of triangles, and compose with the
previous map of triangles, to obtain the diagram

$$
\xymatrix{
T_0 \ar[r] \ar[d] & Y \ar[r]^{c} \ar@{=}[d] & M \ar[r] \ar[d]^{b} & \Sigma T_0
\ar[d]^{\Sigma a} \\
X \ar[r]^{f} \ar[d] & Y \ar[d]^{p}  \ar[r] &  Z \ar[r] \ar[d]^{d} & \Sigma X
\ar[d] \\
U \ar[r]^{r}  & N \ar[r] & Z' \ar[r] & \Sigma U  
}
$$
The composition $T_0 \to X \to U$ vanishes, since
$U$ is in $\X_T$. 
Hence the composition $M \overset{b}{\to} Z \overset{d}{\to} Z' \to \Sigma U$ also
vanishes.
Now Lemma \ref{chasing} implies that there exist maps $e_1 \colon Y \to U$
and $e_2 \colon M \to N$, such that $p = r e_1 +  e_2 c$.
Since $U$ is in $\X_T$, this implies $\underline{p}$ =
$\underline{e_2} \underline{c}$.

This shows that $\underline{c}$ is a weak cokernel for $f$. 
It is also clear that $\underline{c}$ is an epimorphism,
using the triangle~(\ref{c-tria}).
It then follows from Lemma~\ref{weakandepi}
that $\underline{c}$ is actually a cokernel for $f$.
\end{proof}

\begin{proof}[Proof of Theorem~\ref{preab}]
The additivity of $\C/ \X_T$ follows directly from the additivity of
$\C$. By Lemma \ref{kernelcokernel}, we have that for any map  $f \colon X \to Y$,
the induced map $\underline{f}$ has both a kernel and a cokernel.
\end{proof}

In order to show that $\CX$ is also integral, we need to study its projective
objects.
According to~\cite{maclane}, an object $P$ in a preabelian category (indeed, in
any category) is said to be \emph{projective} if, for any epimorphism
$c \colon B\to C$, any morphism $f \colon P \to C$ factors through $c$:
 
$$\xymatrix{
& P \ar[d]^f \ar@{-->}[dl] \\
B \ar[r]_c & C
}$$

We shall use this definition.
But we note that Rump~\cite[p170]{r01} uses a different definition; the
above diagram should commute only for any cokernel $c$.
Such objects are referred to as \emph{quasi-projectives}
in~\cite[Defn.\ 7.5.2]{osborne} and we shall use this terminology.
Note that the two notions are the same in an abelian category, as then
epimorphisms and cokernels coincide. The dual objects will be referred to
as \emph{quasi-injectives}.

In the following three proofs, we use some arguments based on the proof
of~\cite[Theorem 4.3]{kz}.

\begin{lemma} \label{l:addTprojective}
Every object in $\add T$, when regarded as an object in $\CX$, is projective.
\end{lemma}
\begin{proof}
Let $\underline{f} \colon X \to Y$ be an epimorphism in $\CX$, and
$\underline{u} \colon T_0 \to Y$ any morphism, where
$T_0$ lies in $\add T$.
Completing $f$ to a triangle in $\C$, we have the diagram:
$$\xymatrix{
& T_0 \ar[d]_u \\
X \ar[r]_f & Y \ar[r]_g & Z \ar[r] & \Sigma X
}$$
Since $\underline{f}$ is an epimorphism, by Lemma~\ref{diagonals}
we have that $g$ factors through $\X_T$. Hence $gu=0$, so $u$ factors through
$f$ and thus $\underline{u}$ factors through $\underline{f}$ as required.
\end{proof}

\begin{lemma} \label{l:enoughprojectives}
The category $\CX$ has enough projectives.
\end{lemma}
\begin{proof}
Let $X$ be an object in $\CX$. and let $f \colon T_0\to X$ be a minimal right
$\add T$ approximation of $X$ in $\C$. Complete it to a triangle:
$$\xymatrix{
U \ar[r] & T_0 \ar[r]^f & X \ar[r] & \Sigma U.
}$$
By Wakamatsu's Lemma (see Lemma~\ref{wakamatsu}), $U$ lies in $T^{\perp}$, so
$\Sigma U$ lies in $\Sigma T^{\perp}=\X_T$. Hence, by Lemma~\ref{diagonals},
$\underline{f} \colon T_0 \to X$ is an epimorphism, as required. It now follows from Lemma~\ref{l:addTprojective} that $\CX$ has enough projectives.
\end{proof}

Dually, it can be shown that:

\begin{lemma} \label{l:enoughinjectives}
\begin{enumerate}
\item[(a)] Every object in $\add \Sigma^2 T$, regarded as an object in $\CX$, is injective.
\item[(b)] The category $\CX$ has enough injectives.
\end{enumerate}
\end{lemma}

We recall:

\begin{proposition} \label{p:tobesemi-abelian}
\cite[Cor.\ 2]{r01} \\
If $\A$ is a preabelian category with enough quasi-projectives
(respectively, quasi-injectives), then $\A$ is left (respectively, right) semi-abelian.
\end{proposition}

It already follows from Proposition~\ref{p:tobesemi-abelian}
that $\CX$ is semi-abelian, using Lemmas~\ref{l:enoughprojectives}
and~\ref{l:enoughinjectives}. However, a modification of the
argument in the proof of this result allows us to show:

\begin{proposition} \label{p:tobeintegral}
Let $\A$ be a preabelian category.
\begin{enumerate}
\item Suppose that $\A$ has enough projectives. Then $\A$ is left integral.
\item Suppose that $\A$ has enough injectives. Then $\A$ is right integral.
\item Suppose that $\A$ has enough projectives and enough injectives.
Then $\A$ is integral.
\end{enumerate}
\end{proposition}

\begin{proof}
As remarked above, we use an approach similar to the proof of~\cite[Cor.\ 2]{r01}.
For (a), suppose we are given a pullback diagram:
\begin{equation}
\label{e:apullback}
\xymatrix{
A \ar[r]^{\underline{a}} \ar[d]_{\underline{b}} & B \ar[d]^{\underline{c}} \\
C \ar[r]_{\underline{d}} & D
}
\end{equation}
in $\CX$ with $\underline{d}$ an epimorphism.
Since $\A$ has enough projectives, there is an epimorphism
$\underline{a}' \colon P \to B$ in $\CX$ where $P$ is projective in $\A$.
Since $\underline{d}$ is an epimorphism, there is a map
$\underline{b}' \colon P \to C$ in $\CX$ such that
$\underline{c}\underline{a}'= \underline{d}\underline{b}'$.
Since the diagram~\eqref{e:apullback} is a pullback,
there is a map $\underline{e} \colon P \to A$ such that
$\underline{a}\underline{e}=\underline{a}'$
and $\underline{b}\underline{e}=\underline{b}'$.
Since $\underline{a}'$ is an epimorphism, so is $\underline{a}$,
and (a) follows.
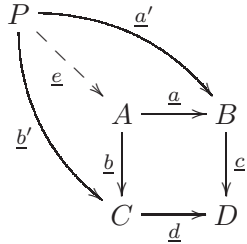
\begin{equation*}
\xymatrix{
P \ar@/^1pc/[drr]^{\underline{a}'} \ar@/_1pc/[ddr]_{\underline{b}'} \ar@{-->}[dr]_{\underline{e}} \\
& A \ar[r]^{\underline{a}} \ar[d]_{\underline{b}} & B \ar[d]^{\underline{c}} \\
& C \ar[r]_{\underline{d}} & D}
\end{equation*}
The proof of (b) is dual to the proof of (a), and (c) follows from (a)
and (b).
\end{proof}

We also remark that, for a semi-abelian category, left integrality is
equivalent to right integrality (see~\cite[Cor. p173]{r01}).

\begin{corollary} \label{c:CXintegral}
The category $\CX$ is integral.
\end{corollary}

\begin{proof}
This follows from Lemmas~\ref{l:addTprojective},~\ref{l:enoughprojectives}
and~\ref{l:enoughinjectives}, together with Proposition~\ref{p:tobeintegral}.
\end{proof}

\section{Localisation}
\label{s:localisation}

Let $\D$ be a category. A class $\R$ of morphisms in $\D$ is said to
admit a \emph{calculus of right fractions}~\cite[I.2]{gz} provided that the following
holds:

\begin{enumerate}
\item[(RF1)] The identity morphisms
of $\D$ lie in $\R$ and $\R$ is closed under
composition.
\item[(RF2)] Any diagram of the form:
$$\xymatrix{
& B \ar[d]^f \\
C \ar[r]_r & D 
}$$
with $r\in \R$ has a completion to a commuting square of the following form:
$$\xymatrix{
A \ar[r]^{r'} \ar[d]_{f'} & B \ar[d]^{f} \\
C \ar[r]_r & D
}$$
with $r'$ in $\R$.
\item[(RF3)]
If $r \colon Y\to Y'$ lies in $\R$ and $f,f' \colon X \to Y$ are maps such that
$rf=rf'$ then there is a map $r' \colon X'\to X$ in $\R$ such that $fr'=f'r'$.
\end{enumerate}

There is a dual set of axioms, (LF1-3), for left fractions.
Let us assume that $\D$ is skeletally small, so that the Gabriel-Zisman
localisation $\D_{\R}$ of $\D$ at $\R$ exists.
In this situation, $\D_{\R}$ has a very nice
description; see~\cite[I.2]{gz} or~\cite[Sect.\ 3]{krause}.
The objects in $\D_{\R}$ are the same as the objects of $\D$. The morphisms
from $X$ to $Y$ are \emph{right fractions}, of the form
$$\xymatrix{
X & A \ar[l]_r \ar[r]^f & Y
}$$ 
denoted $[r,f]_{\text{RF}}$, up to an equivalence relation: two such fractions $[r,f]_{\text{RF}}$ and
$[r',f']_{\text{RF}}$ are equivalent if there is a commutative diagram of the form:
$$\xymatrix{
& A \ar[dl]_r \ar[dr]^f \\
X & A'' \ar[l]_(0.4){r''} \ar[r]^(0.4){f''} \ar[u] \ar[d] & Y \\
& A' \ar[ul]^{r'} \ar[ur]_{f'}
}$$
where $r''$ lies in $\R$.

The composition of two right fractions $[r',f']_{\text{RF}}\circ [r,f]_{\text{RF}}$ is
given by the right fraction $[rr'',f'f'']_{\text{RF}}$ where $r''$, a morphism in $\R$,
and $f''$, a morphism in $\C$, are obtained from an application of axiom (RF2) which
gives rise to the following commutative diagram:
$$\xymatrix{
&& C \ar[dl]_{r''} \ar[dr]^{f''} \\
& A \ar[dl]_r \ar[dr]_f && B \ar[dl]^{r'} \ar[dr]^{f'} \\
X && Y && Z
}$$

The localisation functor from $\D$ to $\D_{\R}$ takes a morphism $f$ to $[id,f]_{\text{RF}}$.
We shall denote this image by $[f]$.
For $r\in \R$, $[r,id]_{\text{RF}}$ is the inverse of $[r]$ (i.e.\ 
the formal inverse adjoined in the localisation). We shall denote it $x_r$.
Thus, every morphism in $\D_{\R}$ has the form $[r,f]_{\text{RF}}=[f]x_r$, where
$f$ is a morphism in $\D$ and $r$ lies in $\R$. Similarly,
if $\D$ satisfies (LF1-3), there is a dual description of $\D_{\R}$ by left
fractions, and so every morphism in $\D_{\R}$ can be written in the form
$[g,s]_{\text{LF}}=x_s[g]$, where $g$ is a morphism in $\D$ and $s$ lies in $\R$.

According to~\cite[p173]{r01}, the following result holds. We include a proof for
the convenience of the reader.

\begin{proposition} \cite[p173]{r01} \label{p:integralfractioncharacterisation}
Let $\A$ be a semi-abelian category. Then $\A$ is integral if and only if the class
$\R$ of regular morphisms in $\A$ admits a calculus of right fractions and a calculus of left fractions.
\end{proposition}
\begin{proof}
We firstly note that it follows from the definitions that $\R$ satisfies (RF1).
Let $r,f,f'$ be as in (RF3) above and suppose that $rf=rf'$. Then $r(f-f')=0$.
Since $r$ is regular, it is a monomorphism, so $f-f'=0$ and $f=f'$. Thus we can just
take $r'$ to be the identity map on $X$ and we see that (RF3) is satisfied.

We will now show that $\A$ is left integral if and only if (RF2) holds.
Suppose first that $\A$ is left integral and we are given a diagram:
\begin{equation}
\xymatrix{
& B \ar[d]^f \\
C \ar[r]_r & D
}
\end{equation}
with $r$ regular.
Let
$$\xymatrix{
A \ar[r]^a \ar[d]_b & C \ar[d]^f \\
B \ar[r]_r & D
}$$
be the pullback of this diagram. Since $r$ is an epimorphism and $\A$ is
left integral, $a$ is also an epimorphism. Since $r$ is a monomorphism,
$a$ is a monomorphism by Lemma~\ref{l:pullbackofmono}. Hence $a$ is also
regular and we see that (RF2) holds.

Conversely, suppose that (RF2) holds and consider a pullback diagram
of the form:
\begin{equation} \label{e:pullback2}
\xymatrix{
A \ar[r]^a \ar[d]_b & B \ar[d]^c \\
C \ar[r]_d & D
}
\end{equation}
with $d$ regular. By (RF2), there is a commuting
diagram
$$\xymatrix{
A' \ar[r]^{a'} \ar[d]_{b'} & B \ar[d]^c \\
C \ar[r]_d & D
}$$
with $a'$ regular. Since the diagram~\eqref{e:pullback2} is a pullback,
we have a map $e \colon A'\to A$ making the diagram:
\begin{equation}
\xymatrix{
A' \ar@/^1pc/[drr]^{a'} \ar@/_1pc/[ddr]_{b'} \ar@{-->}[dr]_e \\
& A \ar[r]^a \ar[d]_b & B \ar[d]^c \\
& C \ar[r]_d & D
}
\end{equation}
commute. Since $a'$ is regular, it is an epimorphism, so $a$ is also
an epimorphism. Again by Lemma~\ref{l:pullbackofmono}, the fact that
$d$ is a monomorphism implies that $a$ is also a monomorphism. Hence
$a$ is regular.
By Proposition~\ref{p:integralcharacterisation}, $\A$ is integral,
hence left integral.

Thus we have seen that $\A$ is left integral if and only if (RF2) holds,
if and only if $\R$ admits a calculus of right fractions.
A similar argument shows that $\A$ is right integral if and only if
$\R$ admits a calculus of left fractions. The result is proved.
\end{proof}

We note that the proof shows that in fact:

\begin{corollary} \label{c:admitcalculus}
Let $\A$ be a semi-abelian category. Then $\A$ is integral if and only if the class
$\R$ of regular morphisms in $\A$ admits a calculus of right fractions
(respectively, a calculus of left fractions).
\end{corollary}
\begin{proof}
If $\A$ is integral then it is left integral.
The proof of Proposition~\ref{p:integralfractioncharacterisation} shows that
then RF1-3 are satisfied by $\R$ and conversely
that if RF1-3 are satisfied then $\A$ is integral. The statement for right fractions
follows, and a dual argument shows the statement for left fractions.
\end{proof}

In the rest of this section, we assume that $\A$ is skeletally small, so that
localisations exist.

\begin{remark} \label{r:inheritsadditivity}
We note that, by~\cite[3.3, Cor.\ 2]{gz}, the localisation of $\A$ at $\R$ in
the situation of Proposition~\ref{p:integralfractioncharacterisation} is
additive, since $\A$ is additive and $\R$ admits a
calculus of right fractions. Furthermore, the localisation functor is additive. See also~\cite[10.3.11]{weibel}.
\end{remark}

\begin{lemma} \label{l:faithfullocalisation}
Let $\A$ be an integral category and let $\R$ be the class of regular
morphisms in $\A$. Then the localisation functor $L \colon \A\to \A_{\R}$ is faithful.
\end{lemma}
\begin{proof}
By Corollary~\ref{c:admitcalculus} (and
recalling~\ref{p:integralimpliessemi-abelian}), $\R$ admits a calculus
of right fractions.
Let $f \colon X \to Y$ be a morphism in $\A$ and suppose that $[f]=0$.
Then we have a commutative diagram:
$$\xymatrix{
& A \ar[dl]_{id} \ar[dr]^f \\
X & A'' \ar[l]_(0.4){r''} \ar[r]^(0.4){f''} \ar[u]^u \ar[d]_v & Y \\
& A' \ar[ul]^{id} \ar[ur]_{0}
}$$
We see that $r''=u=v$ is regular, and $fu=0$. Since $u$ is an epimorphism,
$f=0$ as required.
\end{proof}

We remark that, as a consequence, $[r,f]_{\text{RF}}=[g,s]_{\text{LF}}$ if and only if
$[f]x_r=x_s[g]$, if and only if $[sf]=[gr]$, if and only if $sf=gr$, as noted
in~\cite[p173]{r01}.

We note that:

\begin{lemma} \label{l:epicharacterisation}
Let $\A$ be an integral category and let $\R$ be the class of regular
morphisms in $\A$. Then a morphism $f$ in $\A$ is an epimorphism if and only if
$[f]$ is an epimorphism. It is a monomorphism if and only if $[f]$ is a monomorphism.
\end{lemma}
\begin{proof}
Suppose first that $[f]$ is an epimorphism and $g$ is a morphism in $\A$ for which
$gf=0$. Then $[g][f]=[gf]=0$. Hence $[g]=0$ since $[f]$ is an epimorphism.
By Lemma~\ref{l:faithfullocalisation}, $g=0$, so $f$ is an epimorphism. Conversely,
suppose that $f$ is an epimorphism and that $(x_r[g])[f]=0$ for a morphism $g$ in
$\A$ and a regular morphism $r$ in $\A$. Then $[g][f]=0$, so $[gf]=0$, so by Lemma~\ref{l:faithfullocalisation}, $gf=0$. Hence $g=0$, so $[g]=0$ and therefore
$x_r[g]=0$ as required. So $[f]$ is an epimorphism.
The monomorphism case is proved similarly.
The result is proved.
\end{proof}

\begin{lemma} \label{l:cokercharacterisation}
Let $\A$ be an integral category and let $\R$ be the class of regular
morphisms in $\A$. Let $f,r$ be morphisms in $\A$, with $r$ regular, and let $c$
be a cokernel of $f$ in $\A$. Then $[c]$ is a cokernel of $[f]x_r$ in
$\A_{\R}$. Similarly, if $j$ is a kernel of $f$ in $\A$ then $[j]$ is a kernel
of $x_r[f]$ in $\A_{\R}$.
\end{lemma}
\begin{proof}
We have $[c][f]x_r=[cf]x_r=[0]x_r=0$.
Suppose that $g,s$ are morphisms in $\A$, with $s$ regular, and $(x_s[g])([f]x_r)=0$.
Then $x_s[gf]x_r=0$, so $[gf]=0$, so $gf=0$ in $\A$. Hence $g$ factors through
$c$, so $[g]$ factors through $[c]$, so $x_s[g]$ factors through $[c]$. Hence $[c]$
is a weak cokernel of $[f]x_r$. Since $c$ is an epimorphism in $\A$, it follows
from Lemma~\ref{l:epicharacterisation} that $[c]$ is an epimorphism in $\A$.
Therefore, by Lemma~\ref{weakandepi}, $[c]$ is a cokernel of $[f]x_r$.
The result for kernels is proved similarly.
\end{proof}

We recall that every morphism $f \colon X \to Y$ in a preabelian category $\A$ has a factorisation of the form:
$$\xymatrix{
X \ar@{>>}[r]^(0.35){u} & \coim(f) \ar[r]^{\tilde{f}} & \im(f) \ar@{^{(}->}[r]^(0.6){v} & Y
}$$
and $\A$ is abelian if and only if $\tilde{f}$ is an isomorphism for all morphisms
$f$ in $\A$.

\begin{lemma} \cite[p167]{r01}  \label{l:semi-abelianregular}
Let $\A$ be a preabelian category. Then $\A$ is semi-abelian if and only if
$\tilde{f}$ is regular for all morphisms $f$ in $\A$.
\end{lemma}
\begin{proof}
By definition (see Section~\ref{s:preabelian}),
if $\A$ is semi-abelian then every morphism $f$ has a factorisation of the form
$ip$ where $p=\coim(f)$ and $i$ is a monomorphism.
Comparing this with the factorisation above we see that
$i=v\tilde{f}$
and thus, since $i$ is a monomorphism, so is $\tilde{f}$. Dually, we see that $\tilde{f}$
is an epimorphism, and hence regular. Conversely, suppose that for all morphisms
$f$ in $\A$, $\tilde{f}$ is regular. Then, in the factorisation above,
$v\tilde{f}$ must be
a monomorphism as $v$ and $\tilde{f}$ are.
Dually, $\tilde{f}u$ is an epimorphism and we see that
$\A$ is semi-abelian as required.
\end{proof}

According to~\cite{r01}, we have the following theorem. Again we give
details for the interested reader.

\begin{theorem} \cite[p173]{r01} \label{t:abelianlocalisation}
Let $\A$ be an integral category. Then the localisation
$\A_{\R}$ (if it exists) of $\A$ at the class of regular morphisms
is an abelian category.
\end{theorem}
\begin{proof}
As we have already observed, by~\cite[3.3, Cor.\ 2]{gz}, $\A_{\R}$ is an additive category.
By Lemma~\ref{l:cokercharacterisation}, $\A_{\R}$ is preabelian.
Since $\A$ is semi-abelian (by Proposition~\ref{p:integralimpliessemi-abelian}),
it follows from Lemma~\ref{l:semi-abelianregular} that in the factorisation:
$$\xymatrix{
X \ar@{>>}[r]^(0.35){u} & \coim(f) \ar[r]^{\tilde{f}} & \im(f) \ar@{^{(}->}[r]^(0.6){v} & Y
}$$
of any morphism $f$ in $\A$, $\tilde{f}$ is regular. It is easy to check that
applying the localisation functor to this factorisation gives the corresponding factorisation
of $[f]$.
It follows that for morphisms of the form $\alpha=[f]$ in $\A_{\R}$, $\tilde{\alpha}$
is invertible. Since every morphism in $\A_{\R}$ can be obtained by composing a morphism of
this form with an invertible morphism in $\A_{\R}$, it follows that $\tilde{u}$ is invertible
for all morphisms $u$ in $\A_{\R}$ and hence that $\A_{\R}$ is abelian as required.
\end{proof}

\section{The localisation of $\CX$ is equivalent to $\module \Gamma$.}
\label{s:mainresult}

In this section we will show that $\CXR$ is isomorphic to $\module
\Gamma$, where as before $\Gamma = \End_{\C}(T)^{op}$.

We have seen (Corollary~\ref{c:CXintegral})
that $\CX$ is an integral category. Since we assume $\C$ is skeletally
small, $\CX$ is also skeletally small.
Applying Theorem~\ref{t:abelianlocalisation} to the
integral category $\CX$ we see that $\CXR$ is abelian:

\begin{theorem} \label{t:mainresult1}
Let $\C$ be a skeletally small, Hom-finite, Krull-Schmidt triangulated category with Serre 
duality, containing a rigid object $T$. Let $\X_T$ denote the class of objects $X$ in $\C$ such 
that $\Hom_{\C}(T,X)=0$. Then the class $\R$ of regular morphisms in $\CX$ admits a calculus
of left fractions and a calculus of right fractions.
Furthermore, the localisation $\CXR$
of $\CX$ at the class $\R$ is abelian.
\end{theorem}

\begin{remark} \label{r:inheritskadditivity}
We have seen that the localisation $\CXR$ inherits an additive structure from $\CX$
(see Remark~\ref{r:inheritsadditivity}).
The localisation $\CXR$ also inherits a $k$-additive structure from $\CX$:
a scalar $\lambda$ takes the fraction $[r,f]_{\text{RF}}$ to $[r,\lambda f]_{\text{RF}}$.
It can be checked that this action is well defined and, together with the additive
structure inherited from $\CX$, gives a $k$-additive structure on $\CXR$ for which the
localisation functor is $k$-additive.
\end{remark}

We will show that the projectives in $\CXR$ are the objects in $\add T$ and
that $\CXR$ has enough projectives. From this it will follow that
$\CXR$ is equivalent to $\module \Gamma$.

\begin{lemma} \label{l:regularisomorphism}
Let $\A$ be an additive category and $P$ a projective object in $\A$.
If $r\colon U\to P$ is a regular morphism then it is an isomorphism.
\end{lemma}
\begin{proof}
Since $r$ is an epimorphism, the identity map on $P$ factors through $r$, so there
is a morphism $s\colon P\rightarrow U$ such that $rs=id$.
Then $r(sr-id)=(rs)r-r=0$. Since $r$ is a monomorphism, $sr-id=0$ and it follows
that $r$ is an isomorphism.
\end{proof}

\begin{lemma} \label{l:epiprojectivepreserved}
Let $\A$ be a skeletally small integral category and $\R$ the class of regular
morphisms in $\A$. Suppose that $P$ is a projective object in $\A$. Then $P$, when regarded
as an object in the localisation $\A_{\R}$, is again a projective object.
\end{lemma}
\begin{proof}
Suppose $P$ is projective in $\A$ and we have a diagram:
$$\xymatrix{
& P \ar[d]_{[p]x_s} \\
X \ar[r]_{[f]x_r} & Y
}$$
in $\A_{\R}$ with $p,f,r,s$ morphisms in $\A$, $s \colon U \to P$ and $r$
regular morphisms and
$[f]x_r$ an epimorphism in $\A_{\R}$.
Since $s$ is regular and $P$ is projective in $\A$, $s$ is an isomorphism by
Lemma~\ref{l:regularisomorphism}. Hence $P\simeq U$ in $\A$, so $U$ is also
projective in $\A$.
By Lemma~\ref{l:epicharacterisation}, $f$ is an epimorphism in $\A$,
so $p$ factors through $f$. It follows that
$[p]x_s$ factors through $[f]x_r$ as required.
\end{proof}

\begin{lemma} \label{l:localisationenoughprojectives}
\begin{enumerate}
\item[(a)]
The projectives in $\CXR$ are exactly the objects in $\add T$.
\item[(b)]
The category $\CXR$ has enough projectives.
\end{enumerate}
\end{lemma}
\begin{proof}
By Lemmas~\ref{l:epiprojectivepreserved} and~\ref{l:addTprojective}, the
objects in $\add T$ are projective in $\CXR$.
Let $X$ be an object in $\CX$. Then, by Lemma~\ref{l:enoughprojectives},
there is an epimorphism $\underline{p} \colon T_0\to X$ in $\CX$, where $T_0$
lies in $\add T$, and (b) follows, using Lemma~\ref{l:epicharacterisation}.
If $X$ is projective, the identity map on $X$
factors through $[\,\underline{p}\,]$, so $[\,\underline{p}\,]$ is a
split epimorphism and $X$ is isomorphic to a summand of $T_0$,
hence in $\add T$, and (a) follows.
\end{proof}

\begin{lemma} \label{l:endomorphismalgebraiso}
We have that $\End_{\CXR}(T)\simeq \End_{\C}(T)$.
\end{lemma}
\begin{proof}
The localisation functor induces a morphism
$$\varphi \colon\End_{\CX}(T)\to \End_{\CXR}(T).$$
If $[\underline{f}]x_{\underline{r}}$ is an arbitrary element of $\End_{\CXR}(T)$ with
$\underline{f},\underline{r}$ morphisms in $\CX$ and $\underline{r}$ regular,
then $\underline{r}$ is an isomorphism in $\CX$ by Lemma~\ref{l:regularisomorphism}, so $x_{\underline{r}}=[r^{-1}]$ and we see that $\varphi$ is
surjective. By Lemma~\ref{l:faithfullocalisation}, it is also injective, so
$$\End_{\CX}(T)\simeq \End_{\CXR}(T).$$
The result follows, since the only homomorphisms from $T$ to objects
in $\X_T$ are zero by definition.
\end{proof}

\begin{theorem} \label{t:mainresult2}
Let $\C$ be a skeletally small, Hom-finite, Krull-Schmidt triangulated category with Serre
duality, containing a rigid object $T$. Let $\X_T$ denote the class of objects $X$ in $\C$ such
that $\Hom_{\C}(T,X)=0$. Let $\R$ denote the class of regular morphisms in $\CX$
and $\CXR$ the localisation of the integral category $\CX$ at $\R$.
Then
$$\CXR \simeq \module \End_{\C}(T)^{\op}.$$
\end{theorem}
\begin{proof}
By Theorem~\ref{t:mainresult1}, $\CXR$ is an abelian category.
By Lemma~\ref{l:localisationenoughprojectives},
$\CXR$ has enough projectives, given by the objects in $\add T$.
The result follows, with an equivalence being given by the functor $\Hom_{\CXR}(T,\, -)$,
noting that $T$ is a projective generator for $\CXR$.
\end{proof}

We note that, by Remark~\ref{r:inheritskadditivity}, $\CXR$ inherits a $k$-additive
structure from $\CX$. It is easy to see that the above equivalence preserves this
structure (as well as the abelian structure).

\section{Cotorsion pairs}
\label{s:cotorsionpairs}

We recall the notion of a \emph{cotorsion pair} in a triangulated category,
considered in Nakaoka~\cite{nakaoka}. By~\cite[2.3]{nakaoka} this can be
defined as a pair $(\U,\V)$ of full additive subcategories satisfying
\begin{enumerate}
\item[(a)] $\U^{\perp}=\V$;
\item[(b)] $\V^{\perp}=\U$;
\item[(c)] For any object $C$, there is a (not necessarily unique) triangle:
$$U \to C \to \Sigma V \to \Sigma U,$$
with $U\in \U$ and $V\in \V$.
\end{enumerate}

Nakaoka points out that $(\U,\V)$ is a cotorsion pair in this sense if and
only if $(\U,\Sigma \V)$ is a torsion theory in the sense of~\cite[2.2]{iy}.

If $\C$ is a triangulated category as in Section~\ref{s:notation} and $T$ is
a rigid object in $\C$ then $(\add T,T^{\perp})$ is a cotorsion pair
(e.g.\ see~\cite[Sect.\ 6]{bm10}).
One might then ask whether Theorem~\ref{preab} can be generalised to
this set-up. 
However, it is easy to see that this cannot be the case.
Consider a triangulated category $\C$, satisfying our usual
assumptions. Assume that $\C$ has a non-zero nonisomorphism
between two indecomposables.
Then this map does not have a cokernel or kernel in $\C$.
But the pair $(\U,\V) = (\C,0)$ is
clearly a cotorsion pair. This gives many examples of cotorsion
pairs $(\U,\V)$ such that $\C/\V$ is not preabelian.

More interesting examples where $\C/\V$ is not preabelian also exist.
Let $\C$ be the cluster category associated to the quiver:
$$1 \leftarrow 2 \leftarrow 3.$$
For a vertex $i$, let $P_i$ (respectively, $I_i$, $S_i$)
denote the corresponding indecomposable projective (respectively,
injective, simple) module.
Let $\U$ be the additive subcategory whose indecomposable objects are
$P_2,P_3$ and $\Sigma P_3$. It is easy to check that
$\U^{\perp}$ is the additive subcategory with indecomposables given by
$P_1,P_2$ and $S_2$ and that $(\U,\U^{\perp})$ is a cotorsion pair.
Note that the torsion pair $(\U,\Sigma\U^{\perp})$ appears in~\cite{hjr}.

Let $f$ be a non-zero map from $P_3$ to $I_2$. Suppose that
$c\colon I_2\rightarrow C$ is a cokernel of $f$ in $\C/\U^{\perp}$.
Since the only non-zero
maps $g \colon I_2\rightarrow Y$ with $Y$ indecomposable such that $gf=0$ have
$Y=\Sigma P_1$ or $Y=\Sigma P_2$, it follows that $C$ is a direct sum of
copies of $\Sigma P_1$ and $\Sigma P_2$.
Then $dc=0$ for any non-zero map $d$ from $C$ to $\Sigma P_3$, a contradiction
to the fact that $c$ is an epimorphism.

\section{The functor $\Hom_{\C}(T, - ) \colon \C \to \module  \Gamma$}
\label{s:diagrams}
Let $T$ denote a rigid object in a triangulated category
$\C$, where $\C$ satisfies the same properties as earlier, and let
$\Gamma= \End_{\C}(T)^{op}$.
We have seen that we can obtain $\module \Gamma$ in a process consisting
of two steps: first forming the 
preabelian factor category $\C/\X_T$, and then localising this category with
respect to the class of regular morphisms. 

In \cite{bm10} we considered the functor $H = \Hom_{\C}(T, - ) \colon \C
\to \module \Gamma$. Let  
$\S= \S_T$ be the collection of maps $f \colon X \to Y$ in $\C$ with the property
that in
the induced triangle  $$\Sigma^{-1} Z  \overset{h}{\to}  X \overset{f}{\to} Y \overset{g}{\to} Z,$$
both $g$ and $h$ factor through $\X_T$. Let $L_{\S} \colon \C \to
\C_{\S}$ denote the Gabriel-Zisman localisation.
We proved in \cite{bm10} that there is an equivalence $G \colon \C_{\S} \to
\module \Gamma$ such that $ G L_{\S} = H$. We also proved that a map $s$
belongs to $\S$ if and only if $H(s)$ is an isomorphism in $\module \Gamma$.

In this section, we point
out that the functor $H$ is actually naturally equivalent to the
composition of the quotient functor $\C \to \C/ \X_T$ and the
localisation functor with respect to regular morphisms.

Consider the set of maps $\S_0 = \{X \amalg U \to X  \mid U \in \X_T, X \in \C
\} $ and the Gabriel-Zisman localisation $L_{\S_0} \colon \C \to \C_{\S_0}$.
The quotient functor $Q \colon \C \to \C/\X_T$ inverts all maps in
$\S_0$, so there is a functor $G_0 \colon  \C_{\S_0} \to \C/\X_T$,
making the following diagram commute.

$$
\xymatrix{
\C \ar [dr]_{L_{\S_0}} \ar[rr]^{Q} & & \C/ \X_T\\
& \C_{\S_0} \ar@{..>}[ur]_{G_0}  &
}
$$

The localisation functor $L_{\S_0}$ has the following elementary properties.

\begin{lemma}\label{elem}
\begin{itemize}
\item[(a)]
For $U$ in $\X_T$, consider the projection map $\pi_X \colon X
\amalg U \to X$. The inverse of $L_{\S_0}(\pi_X)$ is
$L_{\S_0}(\iota_X)$,
where $\iota_X \colon X \to X \amalg U$ is the canonical inclusion
map.
\item[(b)] Let $u,v$ be maps in $\C$ such that $v$ factors through $\X_T$.
Then $L_{\S_0}(u+v) = L_{\S_0}(u)$ in $\C_{\S_0}$.
\end{itemize}
\end{lemma}

\begin{proof}
The proof is identical to the proof of Lemma 3.5 in \cite{bm10}. 
\end{proof}

By construction, $G_0$ is the identity on objects. It is
clear that $G_0$ is full, since $Q$ has this property.

We claim that $G_0$ is also faithful.
Firstly, note that by
Lemma~\ref{elem}(a), $L_{\S_0}$ is full.
So let
$f$ and $f'$ be maps in $\C$ with $G_0L_{\S_0}(f) = G_0L_{\S_0}(f') $.
Then $H(f) = H(f')$, so $f-f'$ factors through $\X_T$ (by~\cite[Lemma 2.3]{bm10}),
and hence, by Lemma~\ref{elem}, we have that
$L_{\S_0}(f) = L_{\S_0}(f' +f -f')  = L_{\S_0}(f')$. 
Hence, we have the following.

\begin{proposition}\label{isocat}
The induced functor $G_0 \colon \C_{\S_0}\rightarrow \CX$
is an isomorphism of categories.
\end{proposition}

By Lemma~\ref{diagonals}, a morphism $\underline{f}$ in $\CX$ is
regular if and only if $f$ lies in $\S$. Combining this with
Proposition \ref{isocat} we see that the image $L_{\S_0}(\S)$
in the preabelian category $\C_{\S_0}$ consists of exactly the regular morphisms.

Let $\R$ denote the regular morphisms in $\C/\X_T$.
By the universal property of localisation, it follows that 
we also get an induced isomorphism of
categories
$$K \colon (\C_{\S_0})_{L_{\S_0}(\S)} \to (\C/\X_T)_{\R}$$
making the following diagram commute.
$$
\xymatrix{
\C/\X_T \ar[rr]^{L_{\R}}  \ar@<1ex>[dd]^{G_0^{-1}} & & (\C/\X_T )_{\R}
\ar@<1ex>[dd]^{K^{-1}} & 
& & & \\
& & & \\
\C_{\S_0}  \ar[rr]^{L_{L_{\S_0(\S)}}}   \ar[uu]^{G_0} & & ({\C_{\S_0}})_{L_{\S_0(\S)}}
\ar[uu]^{K}
}
$$
It is clear that $H$ factors uniquely through $Q$, and hence by the universal
property of localisation, also uniquely through $L_{\R} \colon \CX \to \CXR$,
so we have a commutative diagram of functors
$$
\xymatrix{
& & &\\
\C \ar[r]_(0.4)Q \ar@/^2pc/[rrr]^H & \C/\X_T \ar[r]_(0.4){L_{\R}}  &  (\C/\X_T)_{\R} \ar[r]_(0.4){H'}
& \module \End_{\C}(T)^{\op}  
}
$$

\begin{lemma} \label{l:nateq}
The functor $H'$ is naturally isomorphic to the functor $\Hom_{\CXR}(T,-)$
which gives an equivalence between $\CXR$ and $\module \End_{\C}(T)^{\op}$ in
Theorem~\ref{t:mainresult2}.
\end{lemma}

\begin{proof}
Firstly, we note that the map $\varphi_X:f\mapsto [\underline{f}]$ gives
an isomorphism from $H'(X)=\Hom_{\C}(T,X)$ to $\Hom_{\CXR}(T,X)$ as $\End_{\C}(T)^{\op}$-modules
(arguing as in the proof of Lemma~\ref{l:endomorphismalgebraiso}).
Secondly, let $u=x_{\underline{r}}[\underline{f}]\in \Hom_{\CXR}(X,Y)$ be an arbitrary morphism,
where $f:X\to Z$ and $r:Y\to Z$ are morphisms in $\C$ for some object $Z$ in $\C$ and $\underline{r}$
is regular in $\CX$.
Consider the diagram:
$$\xymatrix{
H'(X) \ar_{H'(u)}[d] \ar^(0.35){\varphi_X}[r] & \Hom_{\CXR}(T,X) \ar^{\Hom_{\CXR}(T,u)}[d] \\
H'(Y) \ar_(0.35){\varphi_Y}[r] & \Hom_{\CXR}(T,Y).
}$$
Let $\alpha\in H'(X)=\Hom_{\C}(T,X)$. Then:
\begin{align*}
H'(u)(\alpha) &= H'(x_{\underline{r}}[\underline{f}])(\alpha) \\
&= H'(x_r)H'([\underline{f}])(\alpha) \\
&= H'(x_r)H(f)(\alpha) \\
&= H'(x_{\underline{r}})(f\alpha) \\
&= \Hom_{\C}(T,r)^{-1}(f\alpha)=g,
\end{align*}
where $f\alpha=rg$.
Hence $\varphi_Y(H'(u)(\alpha))=[\underline{g}]$.
We also have
\begin{align*}
\Hom_{\CXR}(T,u)(\varphi_X(\alpha)) &=
   \Hom_{\CXR}(T,x_{\underline{r}}[\underline{f}])([\underline{\alpha}]) \\
&= x_{\underline{r}}[\underline{f}][\underline{\alpha}] \\
&= x_{\underline{r}}[\underline{r}][\underline{g}] = [\underline{g}],
\end{align*}
so the diagram commutes and the lemma is proved.
\end{proof}

It follows that $H'$ is an equivalence.

We also have a commutative diagram of functors
$$
\xymatrix@C1.3cm{
\C \ar[r]^(0.4){L_{\S_0}} \ar@/_2pc/[rrr]_{L_{\S}} & \C_{\S_0} \ar[r]^(0.4){L_{L_{\S_0(\S)}}}  & (\C_{\S_0})_{L_{\S_0}(\S)} \ar[r]^(0.6){L'}  & \C_{\S}  \\
& & &
}
$$
in which $L'$ is an isomorphism of categories.
This follows from the universal property satisfied by the
localisation functors involved.

Let $G^{-1}$ denote a quasi-inverse of $G$. Summarising, we have:

\begin{proposition}
We have the following diagram of functors.
The diagram commutes, apart from the rightmost square, which commutes only
up to natural isomorphism.
$$
\xymatrix{
& & & & & \\
& & \C/\X_T \ar[rr]^{L_{\R}}  \ar@<1ex>[dd]^{G_0^{-1}} & & (\C/\X_T )_{\R} \ar[r]^(0.4){H'}_(0.4){\simeq}
\ar@<1ex>[dd]^{K^{-1}} & 
\module \End_{\C}(T)^{\op}  \ar@<1ex>[dd]^{G^{-1}} \\
\C \ar[urr]_(0.4){Q} \ar[drr]^(0.4){L_{\S_0}}  \ar@/^5pc/[urrrrr]^{H}
\ar@/_5pc/[drrrrr]^{L_{\S}} & & & & & \\
& & \C_{\S_0}  \ar[rr]^{L_{L_{\S_0}(\S)}}   \ar[uu]^{G_0} & & ({\C_{\S_0}})_{L_{\S_0(\S)}}
\ar[r]^(0.6){L'}_(0.6){\cong}  \ar[uu]^K  & \C_{\S}  \ar[uu]^G  \\
& & & & &
}
$$
\end{proposition}

\begin{proof}
We have checked above that the diagram commutes apart from the rightmost square.
We recall that, for a localisation functor $L$ and two functors
$J,J'$ composable with it, $JL=J'L$ implies that $J=J'$, by the
universal property. We have that
$$GL'L_{L_{\S_0}(\S)}L_{\S_0}=GL_{\S}=H=H'L_{\R}Q=H'L_{\R}G_0L_{\S_0}=H'KL_{L_{\S_0}(S)}L_{\S_0},$$
so $GL'=H'K$. It follows that $G^{-1}H'$ is naturally equivalent to $L'K^{-1}$.
\end{proof}

We remark that the fact that $L'$ is an isomorphism implies that $H'$ is an equivalence
of categories, by the commutativity of the right hand square. Thus the equivalence
in Theorem~\ref{t:mainresult2} can also be derived from the fact that $G$ is an equivalence
(i.e.~\cite[Theorem 4.3]{bm10}) together with the above analysis. 

\bigskip

\noindent \textbf{Acknowledgements:}
This work was supported by the Engineering and Physical Sciences Research Council
[grant number EP/G007497/1] and by the NFR [FRINAT grant number 196600].
Some of this work was carried out when both authors were visiting the Mathematical
Research Institute (MFO) in Oberwolfach and when RJM was visiting the Institute for
Mathematical Research (FIM) at the ETH in Zurich.

We would like to thank Henning Krause for his helpful comments.
ABB would like to thank RJM, the Algebra, Geometry and Integrable
Systems Group and the School of Mathematics at the University of Leeds for
their kind hospitality during a visit in September 2010, and
RJM would like to thank ABB and the Department of
Mathematics in Trondheim for their kind hospitality during visits in
August-September 2009 and January 2011.

\end{document}